\documentclass{amsart}

\usepackage{amssymb,amsmath,amscd}

\usepackage[colorlinks,linktocpage]{hyperref}
\hypersetup{linkcolor=[rgb]{0,0,0.715}}
\hypersetup{citecolor=[rgb]{0,0.715,0}}

\newtheorem{thm}{Theorem}[section]
\newtheorem{cor}[thm]{Corollary}
\newtheorem{prop}[thm]{Proposition}
\newtheorem{lem}[thm]{Lemma}

\newtheorem{quest}[thm]{Question}
\newtheorem{obs}[thm]{Observation}
\newtheorem{mainthm}{Theorem}

\theoremstyle{definition}
\newtheorem{defn}[thm]{Definition}

\usepackage{comment}

\theoremstyle{remark}
\newtheorem{rem}[thm]{Remark}
\newtheorem{rems}[thm]{Remarks}

\newcommand{\Ric}{\mathrm{Ric}}
\newcommand{\dimH}{\dim_{\mathcal{H}}}
\newcommand{\dimT}{\dim_{\mathcal{T}}}
\newcommand{\length}{\mathrm{length}}

\makeatletter
\let\c@equation\c@thm
\makeatother
\numberwithin{equation}{section}

\title[]{Nonnegative Ricci curvature, splitting at infinity, and first Betti number rigidity}
\author{Jiayin Pan}
\address[Jiayin Pan]{Department of Mathematics, University of California, Santa Cruz, CA, USA.}
\email{jpan53@ucsc.edu}
\author{Zhu Ye}
\address[Zhu Ye]{School of Mathematical Sciences, Capital Normal University, Beijing, China.}
\email{2210501006@cnu.edu.cn}

\begin{document}
\begin{abstract}
We study the rigidity problems for open (complete and noncompact) $n$-manifolds with nonnegative Ricci curvature. We prove that if an asymptotic cone of $M$ properly contains a  Euclidean $\mathbb{R}^{k-1}$, then the first Betti number of $M$ is at most $n-k$; moreover, if equality holds, then $M$ is flat. Next, we study the geometry of the orbit $\Gamma\tilde{p}$, where $\Gamma=\pi_1(M,p)$ acts on the universal cover $(\widetilde{M},\tilde{p})$. Under a similar asymptotic condition, we prove a geometric rigidity in terms of the growth order of $\Gamma\tilde{p}$. We also give the first example of a manifold $M$ of $\mathrm{Ric}>0$ and $\pi_1(M)=\mathbb{Z}$ but with a varying orbit growth order.
\end{abstract}

\maketitle

\section{Introduction}
	
The relationship between Ricci curvature and the first Betti number is an important topic in Riemannian geometry. In \cite{bochner1,bochner2}, Bochner proved that any closed $n$-manifold $M$ with $\mathrm{Ric}\ge 0$ has first Betti number estimate $b_1(M)\le n$; moreover, equality holds if and only if $M$ is isometric to a flat torus $\mathbb{T}^n$. This classical result was later extended to almost nonnegative Ricci curvature and quantitative rigidity; see the works by Gromov \cite{Gromov}, Gallot \cite{gallot}, and Colding \cite{colding}.

For open (complete and noncompact) manifolds with $\mathrm{Ric}\ge 0$, Anderson proved that if the manifold $M$ has volume growth order at least $k$, that is, there is a point $p\in M$ and some constant $c>0$ such that $\mathrm{vol}(B_R(p))\ge cR^k$ for all $R\ge 1$, then $b_1(M)\le n-k$ \cite{anderson}. Together with the result that $M$ has at least linear volume growth by Yau \cite{yau2}, we see the estimate $b_1(M)\le n-1$. In a recent work by the second-named author \cite{Ye}, the rigidity problem for $b_1(M)=n-1$ was considered: 

\begin{thm}\cite{Ye}\label{thm:codim1}
    Let $M^n$ be an open manifold with $\mathrm{Ric}\ge 0$. Then $b_1(M)=n-1$ if and only if $M$ is flat and either isometric to $\mathbb{R}^{1}\times \mathbb{T}^{n-1}$ or diffeomorphic to $\mathbb{M}^2 \times \mathbb{T}^{n-2}$, where $\mathbb{M}^2$ is the open M\"obius band.
\end{thm}

One of the main purposes of this paper is to give an asymptotic condition in terms of $k$ that not only implies $b_1(M)\le n-k$, but also has rigidity when equality holds. 

\begin{rem}
Let us remark that the volume growth order at least $k$ condition, which was considered by Anderson \cite{anderson}, clearly does not have a rigidity counterpart. In fact, we have the following examples.\\
(1) By taking the metric product of a flat torus $\mathbb{T}^{n-k}$ with an open manifold $N^k$ with positive sectional curvature and Euclidean volume growth, $M^n=\mathbb{T}^{n-k}\times N^k$ has volume growth of order $k$ and $b_1(M)=n-k$, but $M$ is not flat.\\
(2) In \cite{nabonnand}, Nabonnand constructed an example $M^4$ of positive Ricci curvature with $b_1(M)=1$ and volume growth  order $3$.
\end{rem}

As the first main result, we use the line splitting structure on asymptotic cones of $M$ to give first Betti number estimate and rigidity. Recall that an \textit{asymptotic cone}, or a \textit{tangent cone at infinity}, of $M$ is the Gromov-Hausdorff limit of
$$(r_i^{-1}M,p)\overset{GH}\longrightarrow (X,x),$$
where $r_i\to\infty$. When $M$ has $\mathrm{Ric}\ge 0$, the limit always exists after passing to a subsequence. We emphasize that $M$ may have non-unique asymptotic cones; in other words, the limit $(X,x)$ may depend on the sequence $r_i$. If $X$ contains a line, then it splits off this line isometrically by Cheeger-Colding splitting theorem \cite{CC97}.

\begin{mainthm}\label{thm:main_b1}
	Let $M^n$ be an open manifold with $\mathrm{Ric}\ge 0$. Suppose that an asymptotic cone of $M$ properly contains a Euclidean $\mathbb{R}^{k-1}$, where $k\ge 1$. Then $b_1(M)\le n-k$; moreover, equality holds if and only if $M$ is flat and isometric to $\mathbb{R}^{k}\times \mathbb{T}^{n-k}$ or $\mathbb{R}^{k-1}\times N^{n-k+1}$, where $N$ is flat and diffeomorphic to $\mathbb{M}^2 \times \mathbb{T}^{n-k-1}$.
\end{mainthm}

\begin{rems}\label{rem:thm_b1}
We give some remarks on Theorem \ref{thm:main_b1}.\\
(1) When $k=1$, the condition holds trivially, thus Theorem \ref{thm:main_b1} recovers Theorem \ref{thm:codim1}. Our proof of Theorem \ref{thm:main_b1} is distinct from that of Theorem \ref{thm:codim1} in \cite{Ye}. \\ 
(2) When $k=n$, Theorem \ref{thm:main_b1} states that $M$ is isometric to $\mathbb{R}^n$.\\
(3) In Theorem \ref{thm:main_b1}, we only require that one of the asymptotic cones properly contains an $\mathbb{R}^{k-1}$. Such an asymptotic cone is isometric to $\mathbb{R}^{k-1}\times Z$, where $Z$ is not a point.
\end{rems}

Let $h_1(M)$ be the dimension of the space of linear growth harmonic functions (including constant functions) on $M$. By the work of Cheeger-Colding-Minicozzi \cite{ccm}, if $h_1(M)=k+1$, then every asymptotic cone of $M$ contains a Euclidean $\mathbb{R}^k$. Since $\mathbb{R}^{k}$ properly contains an $\mathbb{R}^{k-1}$, together with Theorem \ref{thm:main_b1}, we derive the following corollary, which unifies and generalizes the maximal $h_1(M)$-rigidity by Cheeger-Colding-Minicozzi \cite{ccm} and classical $b_1(M)$-rigidity by Bochner \cite{bochner1,bochner2}.

\begin{cor}\label{cor:harm}
   Let $M^n$ be a complete manifold with $\mathrm{Ric}\ge 0$. Then $$h_1(M)+b_1(M)\le n+1;$$
   moreover, equality holds if and only if $M$ is flat and isometric to  $\mathbb{R}^{k}\times \mathbb{T}^{n-k}$, where $k=h_1(M)-1$.
\end{cor}

In general, one cannot expect flatness when $b_1(M)=n-k-1$ in Theorem \ref{thm:main_b1}. For example, let us consider $N^{k+1}=\mathbb{S}^{k-1} \times \mathbb{R}^2$ with the Schwarzschild metric, where $k\ge 3$. Such a metric is Ricci-flat and has a unique asymptotic cone isometric to $\mathbb{R}^{k}$, thus properly contains an $\mathbb{R}^{k-1}$. Then the metric product $M^n=\mathbb{T}^{n-k-1} \times N^{k+1}$ is asymptotically $\mathbb{R}^{k}$ and $b_1(M)=n-k-1$, but $M$ is not flat. In the special case that $M$ has an asymptotic cone as a line, we indeed have rigidity only requiring $b_1(M)=n-2$; more precisely, combining with Theorem \ref{thm:main_b1} ($k=2$), we have:

\begin{mainthm}\label{thm:codim2}
	Let $M^n$ be an open manifold with $\mathrm{Ric}\ge 0$. Suppose that an asymptotic cone of $M$ contains a line and $b_1(M)=n-2$. Then $M$ is flat and isometric to one of the following:\\
    (1) $\mathbb{R}^2 \times \mathbb{T}^{n-2}$;\\
    (2) $\mathbb{R}\times N^{n-1}$, where $N$ is flat and diffeomorphic to $\mathbb{M}^2 \times \mathbb{T}^{n-3}$;\\
    (3) $\mathbb{R}\times K^{n-1}$, where $K$ is a closed and flat manifold with $b_1(K)=n-2$.
\end{mainthm}

The case of containing a Euclidean $\mathbb{R}^2$ is unclear to the authors.

\begin{quest}\label{quest:codim3}
    Let $M^n$ be an open manifold with $\mathrm{Ric}\ge 0$. If an asymptotic cone of $M$ contains a Euclidean $\mathbb{R}^2$ and $b_1(M)=n-3$, then is $M$ flat?
\end{quest}

Let us also mention the question below about $b_1(M)=n-2$, which was raised in \cite{Ye}. It has an affirmative answer when $n=3$ \cite{Liu13} or $\sec\le K$ \cite{anderson}.

\begin{quest}\label{quest:codim2}
   Let $M$ be an open $n$-manifold with $\mathrm{Ric}\ge 0$ and $b_1(M)=n-2$. Does the universal cover $\widetilde{M}$ split isometrically as $\mathbb{R}^{n-2} \times N^2$?
\end{quest}

Besides the first Betti number, which describes the algebraic growth of $H_1(M,\mathbb{Z})$, another notion often considered is the growth of the orbit \cite{anderson,Ye}. Below we always write $\Gamma=\pi_1(M,p)$, which acts isometrically and freely on the Riemannian universal cover $(\widetilde{M},\Tilde{p})$ of $(M,p)$. Let us consider the growth of $\Gamma$-orbit at $\tilde{p}$. Denote by 
$$\Gamma(R)=\{g\in \Gamma| d(\tilde{p},g\tilde{p})\le R\}.$$
We say that $\Gamma$ has \textit{stable orbit growth of order} $l$, if there are $0<c_1<c_2$ such that
$$c_1 R^l\le \#\Gamma(R) \le  c_2 R^l$$
for all $R\ge 1$. When $M$ has $\mathrm{sec}\ge 0$, it follows from Cheeger-Gromoll soul theorem that $\Gamma$ has stable orbit growth with an integer growth order. For $\mathrm{Ric}\ge 0$, the orbit growth order in general are not integers even assuming $\Gamma=\mathbb{Z}$; for example, Nabonnand's examples \cite{nabonnand} and the related constructions considered in \cite{Pan21,PanWei} have $\pi_1(M)=\mathbb{Z}$ and stable orbit growth, but any number $l\ge 1$ could be the orbit growth order given a sufficiently large dimension.

Given dimension $n$, by a volume packing argument by Dirichlet domains and the fact that $M$ has at least linear volume growth, the orbit growth $\#\Gamma(R)$ is always bounded above by $c_2R^{n-1}$; see \cite{anderson}. The rigidity problem when achieving this upper bound was considered by the second-named author \cite{Ye}. 

\begin{thm}\cite{Ye} \label{thm:geo_codim1} Let $M$ be an open manifold with $\Ric\geq0$. Then $\Gamma$ has stable orbit growth of order $n-1$ if and only if $M$ is flat with an $n-1$ dimensional soul.
\end{thm}

Generalizing Theorem \ref{thm:geo_codim1} to cover orbit growth of intermediate order, we prove the following result.

\begin{mainthm}\label{thm:geo_rigid}
Let $M$ be a complete $n$-manifold with $\Ric\geq 0$. Suppose that $\Gamma$ has stable orbit growth of order $l$ and an asymptotic cone of $M$ contains a Euclidean $\mathbb{R}^{k}$, where $k\ge 0$. Then the following hold.\\
(1) $l\leq n-k$; moreover, $l>n-k-1$ if and only if $M$ is isometric to $\mathbb{R}^k\times N^{n-k}$, where $N$ is flat and closed.\\
(2) Assume in addition that $\widetilde{M}$ has Euclidean volume growth. Then $l>n-k-2$ if and only if $M$ is isometric to $\mathbb{R}^k \times N^{n-k}$, where $N$ is flat and either closed or open with an $n-k-1$ dimensional soul.
\end{mainthm}

\begin{rems}\label{rem:thm_geo}
We give some remarks on Theorem \ref{thm:geo_rigid}.\\
(1) Setting $k=0$ in Theorem \ref{thm:geo_rigid}(2), we recover Theorem \ref{thm:geo_codim1}. Indeed, if $\Gamma$ has stable orbit growth of order $n-1$, then by a volume argument $\widetilde{M}$ has Euclidean volume growth and thus Theorem \ref{thm:geo_rigid}(2) applies.\\
(2) Theorem \ref{thm:geo_rigid}(1) is optimal in the sense that $l=n-k-1$ does not imply flatness. For example, let us consider $M=\mathbb{T}^{n-k-1}\times N^{k+1}$, where $N$ has the Schwarzschild metric. Then $\Gamma$ has orbit growth $n-k-1$ and $M$ has a unique asymptotic cone $\mathbb{R}^k$, but $M$ is not flat.\\
(3) Theorem \ref{thm:geo_rigid}(2) is optimal in the sense that $l=n-k-2$ does not imply flatness. For example, let us consider $M=\mathbb{R}^k \times \Sigma^2 \times \mathbb{T}^{n-k-2}$, where $\Sigma$ is a complete surface of positive sectional curvature and Euclidean volume growth.
\end{rems}



Lastly, we construct an example $M$ of $\Ric>0$ and $\pi_1(M)=\mathbb{Z}$ but its orbit growth order can be either $\alpha$ or $\beta$ depending on the scales. In particular, it does not have stable orbit growth.

\begin{mainthm}\label{thm:exmp}
   Given $1<\alpha<\beta$, there is an open $n$-manifold $M$ with $\Ric>0$ and $\pi_1(M)=\mathbb{Z}$, where $n$ is sufficiently large, such that the orbit growth of $\Gamma$ satisfies the following inequalities with constants $c_1,c_2>0$ depending only on $\alpha$ and $\beta$:\\
   (1) There is a sequence $r_i\to \infty$ such that $c_1 r_i^\alpha \le \#\Gamma(r_i)\le c_2 r_i^\alpha$.\\
   (2) There is a sequence $s_i\to \infty$ such that $c_1 s_i^\beta \le \#\Gamma(s_i)\le c_2 s_i^\beta$.
\end{mainthm}

Using a similar construction and considering its universal cover, we also give an example of an open manifold $N$ with $\Ric>0$ such that the Hausdorff dimension of its asymptotic cones can take any real value in an interval $[d_1,d_2]$; see Theorem \ref{thm:vary_Hdim}. Both this and Theorem \ref{thm:exmp} are the first examples with the described properties.

The construction for Theorem \ref{thm:exmp} is developed on Nabonnand's example \cite{nabonnand}, which is a doubly warped product $[0,\infty) \times_f S^k \times_h S^1$ with a decreasing function $h$; also see \cite{Bergery86,Wei88,PanWei} for related constructions. Then the growth rate of $\#\Gamma(R)$ is related to the decay rate of $h$. To achieve a varying growth rate of $\#\Gamma(R)$, we would require $h$ to have a varying decay rate. This involves a careful gluing to ensure the positive Ricci curvature. 

The idea of proving both Theorems \ref{thm:main_b1} and \ref{thm:geo_rigid} is to consider the equivariant asymptotic cone of a suitable covering space of $M$:
$$\begin{CD}(r_i^{-1} \overline{M},\bar{p},\Lambda) @>GH>> (Y,y,G) \\
	@VV\pi V @VV \pi V\\
	(r_i^{-1} M,p) @>GH>> (Z,z),
\end{CD}$$
where $\Lambda$ is the covering group acting isometrically on $\overline{M}$. By the regularity theory of Ricci limit spaces developed by Cheeger-Colding \cite{colding,CC97,CCI}, $\dimH(Y)$, the Hausdorff dimension of $Y$, is at most $n$, and $\dimH(Y)>n-1$ if and only if $\overline{M}$ has Euclidean volume growth. Moreover, if $M$ is non-flat and $\overline{M}$ has Euclidean volume growth, then $Y$ is a metric cone of dimension $n$ with a Euclidean factor of dimension at most $n-2$. Given the assumption that $Z=Y/G$ already contains a Euclidean factor, one expects that the \textit{dimension} of $G$ has an upper bound in some sense. To show the rigidity in Theorems \ref{thm:main_b1} and \ref{thm:geo_rigid}, we prove technical results relating the first Betti number or the stable orbit growth order to certain dimension of the orbit $Gy$. More precisely, the first Betti number $b_1(M)$ gives a lower bound of the topological dimension of the orbit $Gy$ (Corollary \ref{cor:asy_topol_dim}), and the stable orbit growth order equals the Hausdorff dimension of $Gy$ (Proposition \ref{prop:stable_dimH}). When the orbit $Gy$ has a large dimension, $Y$ will not only be a noncollapsing limit and thus a metric cone, but also contain sufficiently many lines such that $\overline{M}$ must be isometric to $\mathbb{R}^n$.

\textit{Acknowledgement.} J. Pan is partially supported by National Science Foundation DMS-2304698 and Simons Foundation Travel Support for Mathematicians. Z. Ye is partially supported by National Natural Science Foundation of China [11821101] and [12271372]. The authors would like to thank Xiaochun Rong for helpful discussions during the preparation of this paper.
 
\tableofcontents

\section{Preliminaries}
In this section, we recall some standard results on Ricci limit spaces and equivariant Gromov-Hausdorff convergence that will be used later. We also recall some results about flat manifolds. 

\subsection{Ricci limit spaces and asymptotic cones} 

Let $(M_i,p_i)$ be a sequence of $n$-manifolds with uniform Ricci lower bound with $p_i\in M_i$. After passing to a subsequence, $(M_i,p_i)$ Gromov-Hausdorff converge to a limit space $(Y,y)$, which we will refer to as a \textit{Ricci limit space}. After passing to a subsequence again, one can obtain a renormalized limit measure $\nu$ on $Y$, which is a Radon measure. The regularity theory of Ricci limit spaces was developed through the seminal works mainly by Cheeger, Colding, and Naber. Below we review the results that will be used later in the paper.

We write $\mathcal{M}(n,-\kappa)$ as the set of all pointed Ricci limit spaces coming from a sequence of complete $n$-manifolds $(M_i,p_i)$ with $\mathrm{Ric}\ge -(n-1)\kappa$. Let $X\in \mathcal{M}(n,-\kappa)$. We say that a point $x\in X$ is \textit{$k$-regular} for some $k\in\mathbb{N}$, if $X$ has a unique tangent cone $\mathbb{R}^k$ at $x$, that is
$$(r_iX,x)\overset{GH}\longrightarrow (\mathbb{R}^k,0)$$
for all $r_i\to\infty$. We denote by $\mathcal{R}_k$ the set of all $k$-regular points in $X$.

\begin{thm}\cite{CC97}\label{pre:split}
    Let $X\in \mathcal{M}(n,0)$. If $X$ contains a line, then $X$ splits isometrically as $\mathbb{R}\times X'$.
\end{thm}

\begin{thm}\cite{coldingnaber}\label{pre:isom_Lie}
   Let $X\in \mathcal{M}(n,-1)$. Then its isometry group $\mathrm{Isom}(X)$ is a Lie group.
\end{thm}

\begin{thm}\cite{coldingnaber}\label{pre:rect_dim}
   Let $(X,x)\in\mathcal{M}(n,-1)$ be a Ricci limit space with a renormalized limit measure $\nu$. Then there is a unique integer $k\in [0,n]$ such that $\nu(X\backslash \mathcal{R}_k)=0$. In particular, $\mathcal{R}_k$ is dense in $X$.
\end{thm}

The integer $k$ in Theorem \ref{pre:rect_dim} is called the \textit{rectifiable dimension} of $X$.

\begin{thm}\cite{CCI}\label{pre:Hdim_gap}
   Let $(X,x)\in\mathcal{M}(n,-1)$ be a Ricci limit space. Then $X$ has Hausdorff dimension either $=n$ or $\le n-1$. Moreover, the Hausdorff dimension equals $n$ if and only if the sequence $(M^n_i,p_i)\overset{GH}\longrightarrow(X,x)$ is non-collapsing, that is, $\mathrm{vol}(B_1(p_i))\ge v>0$.
\end{thm}

We remake that in general, the Hausdorff dimension of a Ricci limit space could be a non-integer and strictly larger than its rectifiable dimension; see the example by Pan-Wei \cite{PanWei}.

For an open (complete and non-compact) $n$-manifold $M$ with nonnegative Ricci curvature and a sequence  $r_i\to\infty$, the corresponding blow-down sequence naturally gives a Ricci limit space after passing to a subsequence
$$(r_i^{-1}M,p)\overset{GH}\longrightarrow (Y,y).$$
We call $(Y,y)$ an \textit{asymptotic cone}, or a \textit{tangent cone at infinity}, of $M$. In general, $(Y,y)$ is not unique and may not be a metric cone. In the case of Euclidean volume growth, more structure of these asymptotic cones can be said.

\begin{thm}\cite{CC97} \label{pre:vol_cone}
  Let $M$ be an open $n$-manifold of $\mathrm{Ric}\ge 0$ and Euclidean volume growth. Then any asymptotic cone $(Y,y)$ of $M$ is a metric cone $C(Z)$ with Hausdorff dimension $n$ and vertex $y$. 
\end{thm}

Since a metric cone $(Y,y)$ in the above theorem satisfies the Cheeger-Colding splitting theorem (Theorem \ref{pre:split}), we can write $(Y,y)=(\mathbb{R}^k\times C(X), (0,x))$, where $\text{diam}(X)<\pi$ and $x$ is the unique vertex of $C(X)$. From volume convergence \cite{colding} and codimension $2$ of the singular sets \cite{CCI}, if the Euclidean factor of the above $Y$ has dimension $\ge n-1$, then $M$ must be isometric to $\mathbb{R}^n$.

\begin{thm}\cite{colding,CCI}\label{pre:flat}
  Let $M$ be an open $n$-manifold of $\mathrm{Ric}\ge 0$ and Euclidean volume growth. Suppose that $M$ has an asymptotic cone $Y$ that  splits off an $\mathbb{R}^{n-1}$ isometrically, then $Y$ and $M$ are isometric to $\mathbb{R}^n$.
\end{thm}

For a metric cone $Y=\mathbb{R}^k \times C(X)$ with $\text{diam}(X)<\pi$, because $C(X)$ does not contain any line, the isometry group of $Y$ also splits:

\begin{prop}\label{pre:isom_split}
    Let $Y=\mathbb{R}^k \times C(X)$ be a metric cone with vertex $y=(0,x)$, where $\text{diam}(X)<\pi$. Then the isometry group of $Y$ splits 
    $$\mathrm{Isom}(Y)\cong \mathrm{Isom}(\mathbb{R}^k) \times \mathrm{Isom}(X).$$
    In particular, for any subgroup $H$ of $\mathrm{Isom}(Y)$, the orbit $Hy$ is contained in $\mathbb{R}^k\times \{x\}.$
\end{prop}

Proposition \ref{pre:isom_split} immediately implies the following: 

\begin{cor}\label{pre:cone_orb_dim}
   Let $Y=\mathbb{R}^k \times C(X)$ be a metric cone, where $\text{diam}(X)<\pi$. Suppose that $\mathrm{Isom}(Y)$ contains a closed subgroup $\mathbb{R}^b$. Then $b\le k$.
\end{cor}

\subsection{Equivariant Gromov-Hausdorff convergence} 
The theory of equivariant Gromov-Hausdorff convergence was developed by Fukaya and Yamaguchi \cite{eGH1,eGH2}. Below, a tuple $(X,p,G)$ will always denote a complete and locally compact length space $X$, a point $p\in X$, and a closed subgroup $G$ of the isometry group of $X$. Given $R>0$, we write   
$$G(R)=\{g\in G\mid d(gp,p)\le R\}.$$
\begin{defn} \cite{eGH1,eGH2}\label{eqvGHapp}
 Let $(X,p,G)$ and $(Y,q,H)$ be two spaces. Given an $\epsilon>0$, an $\epsilon$-equivariant Gromov-Hausdorff approximation ($\epsilon$-eGHA) from $(X,p,G)$ to $(Y,q,H)$  is a tuple of maps $(f,\phi,\psi)$ 
 $$f:B_{\epsilon^{-1}}(p)\rightarrow B_{\epsilon^{-1}}(q),\quad \phi:G(\epsilon^{-1})\rightarrow H(\epsilon^{-1}), \quad \psi:H(\epsilon^{-1})\rightarrow G(\epsilon^{-1}),$$  
such that the following hold:\\
(1) $f(p)=q$;\\
(2) the $\epsilon$-neighborhood of $f(B_{\epsilon^{-1}}(p))$ contains $B_{\epsilon^{-1}}(q)$;\\
(3) $|d(x_1,x_2)-d(f(x_1),f(x_2))|\leq \epsilon$ for all $x_1,x_2\in B_{\epsilon^{-1}}(p)$;\\
(4) if $g\in G(\epsilon^{-1})$ and $x,gx\in B_{\epsilon^{-1}}(p)$, then $d(f(gx),\phi(g)f(x))\leq \epsilon$;\\
(5) if $h\in H(\epsilon^{-1})$ and $x, \psi(h)x\in B_{\epsilon^{-1}}(p)$, then $d(f(\psi(h)x),hf(x))\leq \epsilon$.\\
We say that a sequence $(X_i,p_i,G_i)$ equivariant GH converges to $(Y,y,H)$, if there are $\epsilon_i\to 0$ and $\epsilon_i$-eGHA from  $(X_i,p_i,G_i)$ to $(Y,y,H)$ for every $i$.
\end{defn}

\begin{thm} \cite{eGH1,eGH2}
Let $(X_i,p_i)\rightarrow (Y,y)$ be a pointed Gromov-Hausdorff convergent sequence, and for each $i$ let $G_i$ be a closed subgroups of $\mathrm{Isom}(X_i)$. Then  \\
(1) after passing to a subsequence, we can find a closed subgroup $H$ of $\mathrm{Isom}(Y)$, such that we have pointed equivariant GH convergence:
$$(X_i,p_i,G_i)\xrightarrow{GH} (Y,y,H).$$
(2) the corresponding sequence of quotient spaces converges:
$$(X_i/G_i, \bar{p}_i)\xrightarrow{GH} (Y/H,\bar{y} )$$
\end{thm}

\subsection{Flat manifolds}

Lastly, we review some results on flat manifolds.

\begin{prop}\cite[Proposition 4]{Ye} \label{open flat} Let $M$ be an open flat $n$-manifold and let $\Gamma=\pi_1(M,p)$. Then the following hold.\\
(1) $\Gamma$ has stable orbit growth of order $k$, where $k$ is the dimension of a soul of $M$.\\
(2) If $\mathbb{T}^{n-1}$ is a soul of $M$, then $M$ is either isometric to $\mathbb{R}\times \mathbb{T}^{n-1}$ or diffeomorphic to $\mathbb{M}^2\times \mathbb{T}^{n-2}$.    
\end{prop}
We also note that if $M$ is a closed flat $n$-manifold, then $\Gamma$ has stable orbit growth of order $n$.
\begin{prop}\label{pre:n-1_flat}
   Let $K$ be a closed $n$-manifold with $\mathrm{Ric}\ge 0$. If $b_1(K)=n-1$, then $K$ is flat.
\end{prop}

\begin{proof}
   This is a direct consequence of Cheeger-Gromoll splitting theorem \cite{CG_split}. We include a brief proof here for readers' convenience.


   Recall that $b=b_1(K)$ is the rank of the abelian group $H_1(K,\mathbb{Z})=\Gamma/[\Gamma,\Gamma]$, where $\Gamma=\pi_1(K,x)$. $H_1$ acts isometrically, freely, and discretely on the intermediate cover $\overline{K}:= \widetilde{K}/[\Gamma,\Gamma]$. By the splitting theorem \cite{CG_split} and the compactness of $K$, the intermediate cover $\overline{K}$ splits isometrically as $\mathbb{R}^l \times \overline{N}$, where $l\le n$ and $\overline{N}$ is compact. Because $\mathrm{Isom}(\mathbb{R}^l\times \overline{N})\cong \mathrm{Isom}(\mathbb{R}^l)\times \mathrm{Isom}(\overline{N})$ contains a closed subgroup $\mathbb{Z}^b\leq H_1$, we see that $b\le l\le n$.
	
    Now we assume that $b=n-1$. Then $l\ge n-1$. As a result, the universal cover $\widetilde{K}$ splits off $\mathbb{R}^{n-1}$ isometrically. Hence $\widetilde{K}$ is isometric to $\mathbb{R}^n$ and $K$ is flat.
\end{proof}

\section{First Betti number and rigidity}

In this section, we study the first Betti number rigidity and prove Theorems \ref{thm:main_b1}, \ref{thm:codim2}. We start with a result on the equivariant GH convergence of denser and denser $\mathbb{Z}^b$-action. This will be later applied to asymptotic cones of an open manifold $M$ or tangent cones of a Ricci limit space at a point.

\begin{lem}\label{lem:topol_dim}
    Let $(X_i,p_i)\in \mathcal{M}(n,-1)$ be a Gromov-Hausdorff convergent sequence of Ricci limit spaces
    $$(X_i,p_i)\overset{GH}\longrightarrow (Y,y).$$
    Suppose that for each $i$, we have a closed subgroup $\Gamma_i=\mathbb{Z}^b\leq \mathrm{Isom}(X_i)$ with generators $\{\gamma_{i,1},...,\gamma_{i,b}\}$ such that $d(\gamma_{i,j}p_i,p_i)\to 0$ as $i\to\infty$, where $j=1,...,b$. After passing to a subsequence, we consider the convergence $$(X_i,p_i,\Gamma_i)\overset{GH}\longrightarrow (Y,y,H).$$
    Then the limit group $H$ contains a closed subgroup $\mathbb{R}^b$.
\end{lem}

In particular, the limit orbit $Hy$ has topological dimension at least $b$ since the closed subgroup $\mathbb{R}^b$ must act freely.

\begin{proof}[Proof of Lemma \ref{lem:topol_dim}]
   We note that each $\Gamma_i$ must act freely on $X_i$; otherwise, $\Gamma_i=\mathbb{Z}^b$ would have a nontrivial compact subgroup as the isotropy subgroup at some point. Since $\Gamma_i$ is abelian, by Theorem \ref{pre:isom_Lie} its limit $H$ is an abelian Lie group. Hence its connected component subgroup $H_0$ is isomorphic to $\mathbb{R}^l \times \mathbb{T}$, where $\mathbb{T}$ is a torus. We show that $l\ge b$ by induction in $b$. 

   We first prove the base case $b=1$. We shall prove that for any $d>0$, there exists an orbit point $hy\in Hy$ with $d(hy,y)=d$ and a compact connected subset $S\subseteq Hy$ containing both $y$ and $hy$. This property implies that $H_0$ contains a closed $\mathbb{R}$-subgroup.

   Let $d>0$ and let $\gamma_i$ be a generator of $\Gamma_i=\mathbb{Z}$. For each $i$, because $\Gamma_i=\mathbb{Z}$ is a closed subgroup of $\mathrm{Isom}(M_i)$, we know that $\Gamma_i$ acts discretely on $M_i$; in particular, $d(\gamma_i^m p_i,p_i)\to \infty$ as $m\to\infty$. For each $i$, we pick $m_i\in\mathbb{Z}_+$ by
   $$m_i= \inf\{ m\in\mathbb{Z}_+ | d(\gamma_i^m p_i,p_i)\ge d \}.$$
   Let us consider a sequence of subsets in $X_i$:
   $$S_i=\{ p_i,\gamma_ip_i,...,\gamma_i^{ m_i}p_i \}\subseteq X_i.$$
   After passing to a subsequence, we have convergence
   $$(X_i,p_i,S_i,\gamma_i^{m_i},\Gamma_i)\overset{GH}\longrightarrow (Y,y,S,h,H).$$
   Because $d(\gamma_ip_i,p_i)\to 0$, the above limit space satisfies
   $$ d(hy,y)=d,\quad  \{y,hy\}\subseteq S \subseteq \overline{B_d}(y) \cap Hy.$$
   It remains to show that $S$ is connected. Suppose not, then by compactness of $S$, we can 
   find two disjoint compact subsets of $S$, say $\mathcal{C}_1$ and $\mathcal{C}_2$, such that
   $$y\in \mathcal{C}_1,\quad  \mathcal{C}_1\cup \mathcal{C}_2=S,\quad d(\mathcal{C}_1,\mathcal{C}_2)=\delta>0.$$
We choose a point $z\in\mathcal{C}_2$. Then $p_i\to y$, and there is $0\le t_i\leq m_i$ such that $\gamma_i^{t_i}p_i\to z$. Since $d(p_i,\gamma_ip_i)\to 0$ as $i\to\infty$, we can find $1<l_i<t_i$ such that $\gamma^{l_i}p_i \to z'$ with
$d(\mathcal{C}_1,z')=\delta/2$. Now we have a point $z'\in S$ but $z'\notin \mathcal{C}_1\cup\mathcal{C}_2$, a contradiction. This proves the base case $b=1$.


Next, suppose that the statement holds for $b$, we shall prove it for $b+1$. For each $i$, we set a subgroup $\Lambda_i=\langle \gamma_{i,1},...,\gamma_{i,b} \rangle \leq \Gamma_i$. For any $d>0$, because $\Gamma_i$ is a discrete subgroup of $\mathrm{Isom}(X_i)$, we can choose $m_i\in \mathbb{Z}_+$ by
   $$m_i=\inf \{ m\in \mathbb{Z}_+\ |\ d( \gamma_{i,b+1}^{m}\Lambda_i p_i, p_i)\ge d \}.$$
Let us consider the subset
   $$S_i=\{ \Lambda_ip_i,\gamma_{i,b+1}\Lambda_ip_i,...,\gamma_{i,b+1}^{ m_i}\Lambda_ip_i \}.$$
   We choose an element $g_i \in \gamma_{i,b+1}^{m_i}\Lambda_i$ such that
$$d(g_ip_i,p_i)=d(\gamma_{i,b+1}^{m_i}\Lambda_ip_i,p_i)\ge d.$$
   After passing to a subsequence, we have convergence
   $$(X_i,p_i,S_i,g_i,\Lambda_i,\Gamma_i)\overset{GH}\longrightarrow (Y,y,S,h,L,H).$$
   It follows from the construction that
   $$d(hy,y)=d(hy,Ly)=d,\quad d(s,Ly)\le d \text{ for all } s\in S.$$
   Let $\pi_H:H\to H/L$ and $\pi_Y:Y\to Y/L$ be the quotient maps.  By the same argument as before, $\pi_H(S)\bar{y}$, where $\bar{y}=\pi_Y(y)$, is a compact and connected subset of $\overline{B_d}(\bar{y})$; moreover $\pi_H(S)\bar{y}$ contains two points $\pi_H(g)\bar{y}$ and $\bar{y}$ that are exactly $d$-apart. Because $d$ is arbitrary, this proves that $H/L$ contains a closed $\mathbb{R}$-subgroup acting effectively on $Y/L$. This completes the inductive step and thus the proof.
\end{proof}

Two corollaries below follow directly from Lemma \ref{lem:topol_dim}.

\begin{cor}\label{cor:asy_topol_dim}
   Let $(M,p)$ be an open $n$-manifold of $\mathrm{Ric}\ge 0$. Suppose that $\mathrm{Isom}(M)$ contains a closed subgroup $\Gamma=\mathbb{Z}^b$. For any $r_i\to\infty$, we consider the convergence $$(r_i^{-1}M,p,\Gamma)\overset{GH}\longrightarrow (Y,y,G).$$
   Then the limit group $G$ contains a closed subgroup $\mathbb{R}^b$. In particular, the orbit $Gy$ has topological dimension at least $b$.
\end{cor}

\begin{cor}\label{cor:tan_topol_dim}
   Let $(X,p)$ be a Ricci limit space. Suppose that $\mathrm{Isom}(X)$ contains a closed subgroup $G=\mathbb{R}^b$. For any $r_i\to\infty$, we consider the convergence $$(r_i X,p,G)\overset{GH}\longrightarrow (Y,y,H).$$
   Then the limit group $H$ contains a closed subgroup $\mathbb{R}^b$. In particular, the orbit $Gy$ has topological dimension at least $b$.
\end{cor}

Now we prove a first Betti number $b_1(M)$ estimate by the rectifiable dimension of some asymptotic cone of $M$.

\begin{thm}\label{thm:b1_rect_dim}
   Let $M$ be a complete $n$-manifold with $\mathrm{Ric}\ge 0$. Suppose that $M$ has an asymptotic cone of rectifiable dimension $k$, then $b_1(M)\le n-k$.
\end{thm}

Theorem \ref{thm:b1_rect_dim} directly implies the $b_1(M)$ estimate in Theorem \ref{thm:main_b1}. In fact, because the $Z$-factor in an asymptotic cone $\mathbb{R}^{k-1}\times Z$ of $M$ is not point, the rectifiable dimension of $\mathbb{R}^{k-1}\times Z$ is at least $k$. Thus $b_1(M)\le n-k$ by Theorem \ref{thm:b1_rect_dim}. We shall see that the proof of Theorem \ref{thm:b1_rect_dim} below also plays an important role in that of the rigidity part of Theorem \ref{thm:main_b1}.

\begin{proof}[Proof of Theorem \ref{thm:b1_rect_dim}]
   Below we write 
   $$b=b_1(M),\quad \Gamma=\pi_1(M,p),\quad \overline{M}=\widetilde{M}/[\Gamma,\Gamma],\quad \Lambda=\Gamma/[\Gamma,\Gamma].$$ Then $\mathrm{Isom}(\overline{M})$ contains $\mathbb{Z}^{b}\leq \Lambda$ as a closed subgroup. By assumption, we have a sequence $r_i\to\infty$ and the convergence
   \begin{equation}\label{cd:rect_asym}
   \begin{CD}(r_i^{-1} \overline{M},\bar{p},\Lambda) @>GH>> (Y,y,G) \\
	@VV\pi V @VV \pi V\\
	(r_i^{-1} M,p) @>GH>> (Z,z),
   \end{CD}
   \end{equation}
   where $Z$ has rectifiable dimension $k$. By Corollary \ref{cor:asy_topol_dim}, $G$ contains a closed subgroup $\mathbb{R}^{b}$. Let ${q}$ be a regular point in $Z$ and let $\bar{q}$ be a lift of $q$ in $Y$. For any sequence $s_i\to\infty$, let us consider
   \begin{equation}\label{cd:rect_tan}
   \begin{CD}
      (s_i Y,\bar{q},G) @>GH>> (\mathbb{R}^k \times Q, (0,\bar{o}), H) \\
      @VV\pi V @VV \pi V \\
      (s_i Z,q) @>GH>> (\mathbb{R}^k,0).
   \end{CD}
   \end{equation}
   By construction and Corollary \ref{cor:tan_topol_dim}, $H$ contains a closed subgroup $\mathbb{R}^{b}$ and the orbit $H\cdot (0,\bar{o})$ is identical to $\{0\}\times Q$. Therefore,  $\mathbb{R}^k\times Q$ has topological dimension at least $k+b$. 
Note that since $\mathbb{R}^k\times Q$ is a tangent cone of $Y$, it is a Ricci limit space coming from $n$-manifolds. So we have
\begin{equation}\label{eq:b1_dim_est}
k+b\leq \dimT(\mathbb{R}^k\times Q)\leq \dimH(\mathbb{R}^k\times Q)\leq n,
\end{equation}
where $\dimH$ and $\dimT$ mean the Hausdorff and topological dimension, respectively. This proves $b_1(M)\leq n-k$.

\end{proof}

\begin{proof}[Proof of Theorem \ref{thm:main_b1}]
   Below we continue to use the notations in the proof of Theorem \ref{thm:b1_rect_dim}. Suppose that $M$ has an asymptotic cone $Z$ that properly contains an $\mathbb{R}^{k-1}$, then in (\ref{cd:rect_asym}), we have
   $$Z=\mathbb{R}^{k-1}\times Z'=Y/G,\quad Y=\mathbb{R}^{k-1}\times Y',$$
   where $Z'$ is not a point and thus $Z$ has rectifiable dimension at least $k$. With Theorem \ref{thm:b1_rect_dim}, it remains to prove the rigidity part when $b=n-k$. 

   From the inequality (\ref{eq:b1_dim_est}), $Y$ has a tangent cone at some point with Hausdorff dimension at least $k+b=n$. We conclude that $Y$ is a noncollapsing Ricci limit space and thus the sequence $r_i^{-1}\overline{M}$ is noncollapsing by Theorem \ref{pre:Hdim_gap}. Hence $\overline{M}$ has Euclidean volume growth. By Theorem \ref{pre:vol_cone}, the asymptotic cone $Y=\mathbb{R}^{k-1}\times Y'$ is a metric cone of dimension $n$, written as $\mathbb{R}^{k-1}\times C(X)$; moreover, the base point $y$ in (\ref{cd:rect_asym}) is a vertex $(0,v)$. 
   
   Because the $\mathbb{R}^{k-1}$-factor in $Y=\mathbb{R}^{k-1} \times C(X)$ is preserved under the quotient map by $G$-action, $G$-acts trivially on the $\mathbb{R}^{k-1}$-factor and the orbit $Gy$ is contained in the $C(X)$-factor. As $G$ contains a closed subgroup $\mathbb{R}^b$, where $b=n-k$, we see from Corollary \ref{pre:cone_orb_dim} that the Euclidean factor in $Y=\mathbb{R}^{k-1} \times C(X)$ has dimension at least $(k-1)+(n-k)=n-1$. By Theorem \ref{pre:flat}, $\overline{M}$ is isometric to $\mathbb{R}^n$ and $M$ is flat. 
   
   Lastly, because $M$ is flat and the asymptotic cone of $M$ properly contains an $\mathbb{R}^{k-1}$, $M$ is isometric to the metric product $ \mathbb{R}^{k-1}\times N^{n-k+1}$ for some open flat manifold $N$ with $b_1(N)=n-k$. Let $S$ be the soul of $N$. $S$ has dimension at most $n-k$ and $b_1(S)=b_1(N)=n-k$. Therefore, $S$ is a flat torus $\mathbb{T}^{n-k}$.
   By Proposition \ref{open flat}(2), $N$ is either isometric to $\mathbb{R}\times \mathbb{T}^{n-k}$ or diffeomorphic to $\mathbb{M}^2\times \mathbb{T}^{n-k-1}$. 
\end{proof}

Next, we prove Theorem \ref{thm:codim2}.

\begin{proof}[Proof of Theorem \ref{thm:codim2}]
   If $M$ has an asymptotic cone that properly contains a line, then the result follows directly from Theorem \ref{thm:main_b1}. It remains to consider the case that $M$ has an asymptotic cone isometric to a line $\mathbb{R}$.

   Let $r_i\to \infty$ be a sequence such that we have convergence
   $$(r_i^{-1} M,p)\overset{GH}\longrightarrow (\mathbb{R},0).$$
   Let us consider two points $y=1$ and $z=-1$ in the limit line $\mathbb{R}$. Then along the sequence $r_i^{-1} M$, we can choose points $y_i,z_i\in r_i^{-1}M$ converging to $y$ and $z$, respectively. By construction,
   $$r_i^{-1} d_M(y_i,z_i)\to 2.$$
   For each $i$, we join $y_i$ and $z_i$ by a minimal geodesic $c_i$. 
   
   If $d_M(c_i,p)<C<\infty$ for some subsequence, then $c_i$ subconverges to a line in $M$. By Cheeger-Gromoll splitting theorem \cite{CG_split}, $M$ splits isometrically as $\mathbb{R}\times K^{n-1}$, where $K$ is a complete $(n-1)$-manifold with $\mathrm{Ric}_K\ge 0$ and $b_1(K)=n-2$. Because $K$ has an asymptotic cone as a single point, we conclude that $K$ must be a closed manifold. Then by Proposition \ref{pre:n-1_flat}, $K$ is flat. 

   If $d_M(c_i,p)\to \infty$ for some subsequence, we set $d_i=d_M(c_i,p)$. Note that $d_i\ll r_i$ since $c_i$ converges to a segment of length $2$ through $0$ when scaling down by $r_i$.
   Now let us consider the convergence $$(d_i^{-1}M,p,c_i)\overset{GH}\longrightarrow (Y,y,c_\infty).$$
   By construction and $d_i\ll r_i$, $c_\infty$ is a line in $Y$ with distance $1$ to $y$. In other words, $M$ has an asymptotic cone $Y$ that properly contains a line. Thus the result follows from Theorem \ref{thm:main_b1} in this case. 
\end{proof}

\section{Orbit growth order and rigidity}
 
In this section, we prove Theorem \ref{thm:geo_rigid}, the orbit growth order rigidity. We always write $\Gamma=\pi_1(M,p)$, which acts isometrically on the Riemannian universal cover $(\widetilde{M},\tilde{p})$. Let $R>0$, we put
$$\Gamma(R)=\{g\in \Gamma| d(\tilde{p},g\tilde{p})\le R\}.$$
For readers' convenience, we recall the notion of stable orbit growth, which was mentioned in the introduction.
\begin{defn}
We say that $\Gamma$ has stable orbit growth of order $k$, if there are constants $0<c_1<c_2$ such that
$$c_1 R^k\leq \#\Gamma(R)\leq c_2R^k$$ 
for all $R\geq 1.$
\end{defn}

The following proposition gives the Hausdorff dimension of limit orbit at the base point exactly as the orbit growth order. We remark that the Hausdorff dimension of the limit orbit at other points could be different.

\begin{prop}\label{prop:stable_dimH}
  Suppose that $\Gamma$ has stable orbit growth of order $k$, where $k\ge 1$. Then for every asymptotic cone $(Y,y,G)$ of $(\widetilde{M},\tilde{p},\Gamma)$, the orbit $G\cdot y$ has Hausdorff dimension $k$.
\end{prop}

Before proving Proposition \ref{prop:stable_dimH}, we introduce some notations and prove a lemma. For a metric space $(X,d)$, a bounded subset $A\subseteq X$, and $\epsilon>0$, we define the capacity
\begin{align*}
	\text{Cap}(A;\epsilon)=\sup\{\ k\mid &\text{ there are $k$ points $x_1,\cdots,x_k\in A$} \\
	&\text{ such that $d(x_i,x_j)\ge \epsilon$ for all $i\neq j$.} \}
\end{align*} 
Let $d$ be the distance on $\widetilde{M}$ and let $r_i\to \infty$ be a sequence. 
We write $\Gamma_i\tilde{p}$ as the orbit $\Gamma\tilde{p}$ equipped with the distance $d_i(g_1\tilde{p},g_2 \tilde{p}):=r_i^{-1}d(g_1\tilde{p},g_2\tilde{p})$. For a point $x\in \Gamma_i\cdot\tilde{p}$ and $r>0$, we write
$$B_r^i(x)=\{y\in  \Gamma_i\tilde{p}\mid  d_i(x,y)\le  r\}$$
as the $r$-ball centered at $x$ in $\Gamma_i\tilde{p}$. Note that $\# B^i_r(x)=\# B^i_r(\tilde{p})$ for all $x\in \Gamma_i\tilde{p}$ because $\Gamma$-action is isometric.

\begin{lem}\label{lem:cont_est}
    Suppose that $\Gamma$ has stable orbit growth of order $k$, where $k\ge 1$. Then for every $i\in \mathbb{N}_+$ and $0<\lambda<R$, it holds that
	\begin{equation}\label{eqkey}
		\frac{c_1}{c_2}\left(\frac{R}{\lambda}\right)^k \leq \mathrm{Cap}(B^i_R(\tilde{p});\lambda)\leq  \frac{3^{k+1}c_2}{c_1} \left(\frac{R}{\lambda}\right)^k.
	\end{equation}	
\end{lem}
\begin{proof}
    The orbit growth condition implies
     \begin{equation}\label{eqtrivial}
		c_1(Rr_i)^k\leq \#B^i_R(\tilde{p}) \leq  c_2(Rr_i)^k.
     \end{equation}	
    We claim the following inequalities:
	\begin{equation}\label{eq:cont_claim}
		\text{Cap}(B^i_{R-\lambda/3}(\tilde{p});\lambda)\cdot \#B^i_{\lambda/3}(\tilde{p})\leq \#B^i_R(\tilde{p}) 
	\leq \text{Cap}(B^i_R(\tilde{p});\lambda)\cdot \#B^i_\lambda(\tilde{p}).
	\end{equation}
	Indeed, given any collection of points $x_1,\cdots,x_k\in B^i_{R-\lambda/3}(\tilde{p})$ such that $d_i(x_a,x_b)\geq \lambda$ for any $x_a\neq x_b$, we have disjoint union $$\bigsqcup\limits_{i=1}^k B^i_{\lambda/3}(x_i)\subseteq  B^i_R(\tilde{p}).$$ 
 This implies that $k\cdot\#B^i_{\lambda/3}(\tilde{p})\leq  \#B^i_R(\tilde{p})$, thus the first inequality in (\ref{eq:cont_claim}). To prove the other one, we fix $y_1,\cdots, y_s\in B^i_R(\tilde{p})$ such that 
 $$ s=\mathrm{Cap}(B^i_R(\tilde{p});\lambda),\quad  d_i(y_a,y_b)\geq \lambda \text{ for any } y_a\neq y_b.$$ 
 Then we have $$\bigcup\limits_{i=1}^s B^i_\lambda(y_i)\supseteq  B^i_R(\tilde{p}).$$ 
 Hence $s\cdot\#B^i_\lambda(\tilde{p})\geq \#B^i_R(\tilde{p})$, the second inequality in (\ref{eq:cont_claim}).
	
  The desired inequality (\ref{eqkey}) follows from (\ref{eqtrivial}) and (\ref{eq:cont_claim}).
\end{proof}

We recall the definition of (spherical) Hausdorff measure. For a metric space $X$, a real number $k\geq 0$, and a  $\delta>0$,  we have 
$$\mathcal{H}^k_{\delta}(X)=\inf\left\{\sum\limits_{i=1}^\infty r_i^k\mid X\subseteq \bigcup\limits_{i=1}^\infty B_{r_i}(x_i), \text{ where } r_i\leq \delta       \right\}.$$
If $X$ is compact, then equivalently we can use finite covers to define $\mathcal{H}^k_\delta$. The $k$-dimensional (spherical) Hausdorff measure of $X$ is given by $\mathcal{H}^k(X):= \lim\limits_{\delta\to 0} \mathcal{H}^k_{\delta}(X)$. The Hausdorff dimension of $X$ is defined by
$$\dimH(X)=\inf\{s\ge 0\mid \mathcal{H}^s(X)=0\}=\sup\{s\ge 0\mid  \mathcal{H}^s(X)>0\}, $$
where we use the convention $\inf\{\emptyset\}=\infty$. 

\begin{proof}[Proof of Proposition \ref{prop:stable_dimH}]
    Since we have pointed Gromov-Hausdorff convergence $$(r_i^{-1}\widetilde{M},\tilde{p},\Gamma_i\tilde{p})\overset{GH}\longrightarrow (Y,y,Gy),$$ the same inequality in the form of (\ref{eqkey}) holds for $G\cdot y$. More precisely, if we write 
    $$B^\infty_R(z)=\{w\in  Gy\mid  d_Y(z,w)\le R\},$$ 
    where $z\in Gy$ and $R>0$, then because $B^\infty_R(z)$ is the GH limit of $B^i_R(z_i)$ for some $z_i\in \Gamma_i\tilde{p}$, we have:
\begin{equation}\label{eq:cont_lim}
	\frac{c_1}{c_2} \left(\dfrac{R}{\lambda}\right)^k\leq \text{Cap}(B^\infty_R(z);\lambda)\leq  \frac{3^{k+1}c_2}{c_1} \left(\dfrac{R}{\lambda}\right)^k.
\end{equation}	
for every $0<\lambda<R$.

The inequality (\ref{eq:cont_lim}) shows that the limit orbit $Gy$ has box dimension $k$. Below, we show that $\mathcal{H}^k$ is indeed locally finite on $Gy$.

Let $\delta>0$. If $l=\text{Cap}(B^\infty_R(y);\delta)$, then by definition there are $x_1,\cdots,x_l\in B^\infty_R(y)$ such that $$B^\infty_R(y)\subseteq \bigcup\limits_{j=1}^l B_{\delta}(x_j).$$ 
Thus by definition of $\mathcal{H}^k_\delta$ and by inequality (\ref{eq:cont_lim}), we have estimate
$$\mathcal{H}^k_\delta (B^\infty_R(y))\leq l\cdot\delta^k\leq 
\frac{3^{k+1}c_2}{c_1}R^k.$$ This shows $\mathcal{H}^k(B^\infty_R(y))\leq \frac{3^{k+1}c_2}{c_1}R^k$ since $\delta$ is arbitrary.

It remains to obtain a lower bound of $\mathcal{H}^k_\delta (B^\infty_R(y))$. Because $B^\infty_R(y)$ is compact, let us consider a finite collection of balls 
$\{B^\infty_{t_1}(x_1),\cdots, B^\infty_{t_m}(x_m)\}$ such that 
$$x_j\in B^\infty_R(y),\quad t_j\le \delta,\quad B^\infty_R(y)\subseteq \bigcup\limits_{j=1}^m B^\infty_{t_j}(x_j).$$ We choose a small number $0<\lambda<\min\{t_1,\cdots,t_m\}$, then by (\ref{eq:cont_lim})
\begin{align*}
    \sum_{j=1}^{m}t_j^k&\geq \left(\sum_{j=1}^{m} \text{Cap}(B^\infty_{t_j}(x_j);\lambda)\right)\frac{c_1}{3^{k+1}c_2}\lambda^k\\
	&\geq \frac{c_1\lambda^k}{3^{k+1}c_2}\text{Cap}(B^\infty_R(y);\lambda) \\
	&\geq \frac{c^2_1}{3^{k+1}c^2_2}R^k.
\end{align*}
Thus $\mathcal{H}^k(B_R^\infty (y))\geq\mathcal{H}_\delta^k(B_R^\infty (y))\geq \frac{c^2_1}{3^{k+1}c^2_2}R^k$. 
\end{proof}

We are ready to prove Theorem \ref{thm:geo_rigid}.

\begin{proof}[Proof of Theorem \ref{thm:geo_rigid}]
Let $r_i\to\infty$ be a sequence such that the limit of $(r_i^{-1}M,p)$ contains an $\mathbb{R}^k$ factor. Passing to a subsequence, we have equivariant GH convergence: 
$$\begin{CD}
	(r_i^{-1}\widetilde{M},\tilde{p},\Gamma) @>GH>> (Y=\mathbb{R}^k\times Y',y,G)\\
    @VV \pi_i V @VV \pi V\\
	(r_i^{-1}M,p) @>GH>> (Z=(\mathbb{R}^k\times Z')=Y/G,z)
 \end{CD}$$	
Because the $\mathbb{R}^k$-factor in $Y$ is preserved under the quotient map by $G$-action, we have $G\cdot y\subseteq \{0^k\}\times Y'$.

By the assumption that $\Gamma$ has stable orbit growth of order $l$ and Proposition \ref{prop:stable_dimH}, it holds that
$$l=\dimH(Gy)\leq \dimH(Y').$$
Hence the Hausdorff dimension of $Y$ satisfies
$$n \ge \dimH(Y)=k+ \dimH(Y')\ge k+l.$$
This proves the estimate $l\le n-k$.

Next, assuming $l>n-k-1$, we prove that $\widetilde{M}$ has Euclidean volume growth. This is an immediate consequence of the Hausdorff dimension estimate above and Theorem \ref{pre:Hdim_gap}. In fact,
$$\dimH(Y)\ge k+l>k+(n-k-1)=n-1.$$
By Theorem \ref{pre:Hdim_gap}, $\dimH(Y)=n$ and thus $\widetilde{M}$ has Euclidean volume growth. 

To prove the flatness part in both (1) and (2) of Theorem \ref{thm:geo_rigid}, it suffices to prove the Claim below.

\textbf{Claim:} If $\widetilde{M}$ has Euclidean volume growth and $l>n-k-2$, then $M$ is flat and thus isometric $\mathbb{R}^k\times N^{n-k}$. In fact, since $\widetilde{M}$ has Euclidean volume growth, by Theorem \ref{pre:vol_cone}, $Y$ is a metric cone with vertex $y$. We further write 
$$Y=\mathbb{R}^k\times Y'= \mathbb{R}^k \times (\mathbb{R}^m \times C(X))=\mathbb{R}^{k+m}\times C(X)$$
and $y=(0^k,0^m,v)$, where $C(X)$ does not contain any lines and $v$ is the unique vertex of $C(X)$. By Proposition \ref{pre:isom_split}, we have inclusion
$$Gy\subseteq \mathbb{R}^{k+m}\times \{v\}.$$
Together with $Gy\subseteq \{0^k\}\times Y'$, we derive
$$Gy \subseteq \{0^k\}\times \mathbb{R}^m\times \{v\}.$$
As $Gy$ has Hausdorff dimension $l$, the Euclidean factor of $Y$ has dimension
$$m+k\ge \dimH(Gy)+k =l+k >(n-k-2)+k=n-2.$$
Thus $Y$ splits off an $\mathbb{R}^{n-1}$-factor. By Theorem \ref{pre:flat}, $\widetilde{M}$ is isometric to $\mathbb{R}^n$. This proves the Claim.

Now we prove the rigidity part in Theorem \ref{thm:geo_rigid}(1). By the Claim above, $M$ is isometric to $\mathbb{R}^k\times N^{n-k}$ for some flat manifold $N$. Note that orbit growth order $l$ is an integer since $M$ is flat, so $l=n-k$. Therefore, $N$ is closed; otherwise, we would have $l\leq n-k-1$ by Proposition \ref{open flat}(1).

Finally, we prove Theorem \ref{thm:geo_rigid}(2). Again by the Claim, if $l>n-k-2$, then $M$ is isometric to $\mathbb{R}^k\times N^{n-k}$ for some flat manifold $N$; moreover, $l$ is an integer. If $l=n-k-1$, then $N$ is an open flat manifold with an $n-k-1$ dimensional soul by Proposition \ref{open flat}(1). If $l=n-k$, then $N$ is closed as discussed above. 
\end{proof}

\section{Examples of $\mathbb{Z}$-actions with varying orbit growth order}

In the work by Pan-Wei \cite{PanWei}, the first examples of Ricci limit spaces with non-integer Hausdorff dimension have been constructed. These examples are the asymptotic cones of the Riemannian universal covers of doubly warped products
\begin{equation}\label{eq:warp}
M=[0,\infty)\times_f S^k \times_h S^1,\quad g=dr^2+f(r)^2ds_k^2 + h(r)^2ds_1^2.
\end{equation}
with
\begin{equation}\label{eq:warp_f}
f'(0)=1,\quad f^{(\text{even})}(0)=0,\quad 0<f'<1,\quad f''<0,
\end{equation}
\begin{equation}\label{eq:warp_h}
h(0)>0,\quad h^{(\text{odd})}(0)=0,\quad h'<0.
\end{equation}
As base manifolds, they were first constructed by Nabonnand \cite{nabonnand} (also see \cite{Bergery86,Wei88}). Since our construction is based on these examples, we first have a brief review of them. The main reference is \cite{PanWei}.

Given $\alpha>0$, we use
$$f(r)=r(1+r^2)^{-\frac{1}{4}},\quad h(r)=(1+r^2)^{-\alpha}.$$
Then
$g_\alpha$ in the form of (\ref{eq:warp}) is a doubly warped product metric on $S^{k+1}\times S^1$. $g_\alpha$ has positive Ricci curvature when $k\geq K(\alpha):=\max\{4\alpha+2,16\alpha^2+8\alpha \}$.

Let $\pi:(\widetilde{M},\Tilde{g}_\alpha)\rightarrow (M,g_\alpha) $ be the Riemannian universal cover, $p\in M$ be a point at $r=0$, $\Tilde{p}$ be a lift of $p$ on $\widetilde{M}$, and $\gamma$ be a generator of $\Gamma:=\pi_1(M,p)\cong\mathbb{Z}$. According to \cite[Lemma 1.1]{PanWei}, the following estimate holds for all $l\geq 9^{1+\frac{1}{2\alpha}}$: 
\begin{equation}\label{eq:length_est}
     C\cdot l^{\frac{1}{1+2\alpha}}-2\leq d(\gamma^l \Tilde{p},\Tilde{p})\leq 9\cdot l^{\frac{1}{1+2\alpha}},
\end{equation}
where $C=2\cdot 9^{-\frac{1}{2\alpha}}$. It is straightforward to deduce from (\ref{eq:length_est}) that the $\Gamma$-action on $\widetilde{M}$ has stable orbit growth of order $1+2\alpha$. Moreover, for any $r_i\to\infty$, passing to subsequence we have the equivariant GH convergence:
\begin{equation}\label{CD:exmp}
\begin{CD}
	(r_i^{-1}\widetilde{M},\Tilde{p},\Gamma) @>GH>> (Y,y,G)\\
    @VV \pi V @VV \pi V\\
	(r_i^{-1}M,p) @>GH>> (X,x)
 \end{CD}
 \end{equation}
 It follows from (\ref{eq:length_est}) that $Gy$ has Hausdorff dimension $1+2\alpha$. Since any points in $Y\backslash Gy$ are $2$-regular, one sees that $\text{dim}_{\mathcal{H}}(Y)=1+2\alpha$ for any $\alpha\geq 1/2$; see \cite{PanWei} for details. The limit space $(Y,y)$ is indeed the $(2\alpha)$-Grushin halfplane, which is an almost Riemannian metric $dr^2+r^{-4\alpha}dv^2$ on the halfplane; see \cite[Remark 3.9]{DHPW}.

 \begin{rem}
 The above $X=[0,\infty)$ is a ray, thus $\text{dim}_{\mathcal{H}}(X)=1$. Then in the diagram (\ref{CD:exmp}), we have
 $$\dimH(Y)=\max\{2,1+2\alpha\}< 2+2\alpha=\dimH(X)+\dimH(Gy).$$
 In particular, the Hausdorff dimension may not work well with the quotient map.
 \end{rem}

We shall use the manifold $(M,g_\alpha)$ above as a model. For $0<\alpha<\beta$, we would like to have the pieces of $(M,g_{\alpha})$ and the pieces of $(M,g_{\beta})$ appear in turn indefinitely as $r\to\infty$. In some suitable scales $r_i\to \infty$, one can only see one of the models, while in another suitable scales $s_i\to \infty$, one can only see the other one. Then it is expected that $\Gamma$-action on the universal cover has orbit growth of order $1+2\alpha$ in some scales, while $1+2\beta$ in some other scales; in particular, $\Gamma$ will not have stable orbit growth. 

Now we start the formal construction of the example for Theorem \ref{thm:exmp}.

As indicated, the example is a doubly warped product (\ref{eq:warp}). We intend to use the same
$$f(r)=r(1+r^2)^{-\frac{1}{4}},$$
as in \cite{PanWei,Wei88}, but the decreasing function $h$ is piecewisely defined and oscillates between
$$h_1(r)=(1+r^2)^{-\alpha},\quad h_2(r)=(1+r^2)^{-\beta},$$
where $\beta>\alpha>0$. The Ricci curvature terms that need checking are
\begin{equation}\label{eq:Ric}
\Ric(\partial_r,\partial_r)=-\dfrac{h''}{h}-k \cdot\dfrac{f''}{f},\quad \Ric(Y,Y)=-\dfrac{h''}{h}-k \cdot\dfrac{f'h'}{fh},
\end{equation}
where $Y$ is a unit vector tangent to $S^1$. The Ricci curvature in the direction of $S^k$ is always positive due to (\ref{eq:warp_f}) and (\ref{eq:warp_h}).

The construction of $h$ will be done in two steps. In the first step, we construct a continuous and piecewise smooth $h$. Then in the second step, we smooth $h$ near the broken points. In both steps, we will further enlarge $k$, the dimension of the sphere factor in the doubly warped product (\ref{eq:warp}), to ensure positive Ricci curvature.

\subsection{Step 1.} We first define a piecewise function $h$ without smoothing it. For convenience, we set
$$g_1(r)=(1+r^2)^{\alpha},\quad g_2(r)=(1+r^2)^\beta.$$
Let us choose $A$ and $B$ such that
$$B>\beta>\alpha>A>0.$$
We set scales 
$$0=R_{1,0}<R_{1,1}=100<R_{1,2}<R_{1,3}<R_{1,4},$$
where $R_{1,2}$, $R_{1,3}$ and $R_{1,4}$ will be determined later. On $[0,R_{1,1}]$, we use $h=h_1$. On $[R_{1,2},R_{1,3}]$, we use $h=h_2$. On $[R_{1,1},R_{1,2}]$, we use a function $h_{1,+}=1/g_{1,+}$ to bridge, where
$$g_{1,+}(r)= C_{1,+} \cdot (1+r^2)^B.$$
The constant $C_{1,+}$ is a small number so that $g_{1,+}$ intersects $g_1$ and $g_2$ at $R_{1,1}$ and $R_{1,2}$, respectively; in other words, we require
$$C_{1,+}(1+R^2_{1,1})^B=(1+R^2_{1,1})^\alpha,\quad C_{1,+}(1+R^2_{1,2})^B= (1+R^2_{1,2})^\beta.$$
This implies that 
$$C_{1,+}=(1+R^2_{1,1})^{\alpha-B}, \quad R_{1,2}=((1+R^2_{1,1})^{\frac{B-\alpha}{B-\beta}}-1)^{\frac{1}{2}}. $$
Next, on $[R_{1,3},R_{1,4}]$, we use a function $h_{1,-}=1/g_{1,-}$, where 
$$g_{1,-}(r)= C_{1,-} \cdot (1+r^2)^A.$$
$C_{1,-}$ is a large number so that $g_{1,-}$ intersects $g_2$ and $g_1$ at $R_{1,3}$ and $R_{1,4}$, respectively. This requires
$$(1+R^2_{1,3})^\beta= C_{1,-}(1+R^2_{1,3})^A,\quad C_{1,-}(1+R^2_{1,4})^A= (1+R^2_{1,4})^\alpha.$$
Hence 
$$C_{1,-}=(1+R_{1,3}^2)^{\beta-A},\quad R_{1,4}=((1+R_{1,3}^2)^{\frac{\beta-A}{\alpha-A}}-1)^{\frac{1}{2}}.$$

Then inductively, we set scales 
$$R_{i+1,0}=R_{i,4}<R_{i+1,1}<R_{i+1,2}<R_{i+1,3}<R_{i+1,4}$$
and repeat the above construction. More specifically, we define 
$$
h(r) = 
\begin{cases}
h_1(r), & \text{on } [R_{i,0},R_{i,1}]; \\
h_{i,+}(r), &\text{on } [R_{i,1},R_{i,2}]; \\
h_2(r),&\text{on } [R_{i,2},R_{i,3}]; \\
h_{i,-}(r), &\text{on } [R_{i,3},R_{i,4}]. 
\end{cases}
$$
where 
$$h_{i,+}(r)= (1+R_{i,1}^2)^{B-\alpha}(1+r^2)^{-B},\quad h_{i,-}(r)=(1+R_{i,3}^2) ^{A-\beta}(1+r^2)^{-A}.$$ 
We require that 
$$(1+R_{i,1}^2)^{B-\alpha}=(1+R_{i,2}^2)^{B-\beta},\quad (1+R_{i,3}^2)^{A-\beta}=(1+R_{i,4}^2)^{A-\alpha}$$
so that $h$ is a continuous and decreasing function on $[0,\infty)$. For each $i$,
$$R_{i,2}=((1+R_{i,1}^2)^{\frac{B-\alpha}{B-\beta}}-1)^{\frac{1}{2}}$$
is determined by $R_{i,1}$, and 
$$R_{i,4}=((1+R_{i,3}^2)^{\frac{\beta-A}{\alpha-A}}-1)^{\frac{1}{2}}$$
is determined by $R_{i,3}$. We also require that $R_{i,3}=5R_{i,2}^2$ and $R_{i+1,1}= 5R_{i,4}^2$ for later use. Then the function $h$ is completely determined by $\alpha$ and $\beta$.
 
Because the terms in Ricci curvature (\ref{eq:Ric}) involving $h$ always appear as $h''/h$ or $h'/h$, the constants $(1+R_{i,1}^2)^{B-\alpha}$ and $(1+R_{i,3}^2) ^{A-\beta}$ in $h_{i,\pm}$ do not contribute to the Ricci curvature. In particular, when $k>K(B)=\max\{ 4B+2, 16B^2+8B \}$, the metric has positive Ricci curvature wherever $h$ is smooth.

Let us also point out that the choice of $k>K(B)$ assures that the doubly warped product (\ref{eq:warp}) has positive Ricci curvature with $h$ as one of $h_1$, $h_2$, $h_{i,+}$, or $h_{i,-}$.

\subsection{Step 2.} Next, we smooth $h$ around $R_{i,j}$ by cutoff functions while preserving positive Ricci curvature.

We need an observation from the Ricci curvature calculation (\ref{eq:Ric}).

\begin{obs}\label{obs:Ric}
Suppose that a doubly warped product (\ref{eq:warp}) with (\ref{eq:warp_f}) and (\ref{eq:warp_h}) has positive Ricci curvature on an interval $I$. We replace $h$ by a new function $h_{new}$ satisfying
\begin{equation}\label{eq:new}
h'_{new}<0,\quad \left|\dfrac{h'_{new}}{h_{new}}\right|>c \left|\dfrac{h'}{h}\right|,\quad \dfrac{h''_{new}}{h_{new}}<C \dfrac{h''}{h}
\end{equation}
on $I$, where $c,C>0$ are constants. Then the new doubly warped product with $h_{new}$ has positive Ricci curvature on $I$ if we further enlarge $k$.
\end{obs}

The smoothing near $R_{i,j}$, where $i\in \mathbb{N}_+$ and $j\in\{1,2,3,4\}$, are all similar. Below we provide the details for $R_{i,1}$ and give brief indications for other $R_{i,j}$.

For convenience, we set $R=R_{i,1}\ge 100$. All positive constants $c_\star$ below may depend on $\alpha,\beta,A,B$, but are independent of $R$. Around $R$, we have two functions
$$h_1=(1+r^2)^{-\alpha},\quad h_{i,+}=(1+R^2)^{B-\alpha}(1+r^2)^{-B}$$
intersecting at $r=R$. Let $\phi(r)$ be a smooth cutoff function such that 
$$\phi\equiv 1 \text{ on } [0,1.01R],\quad \phi \equiv 0 \text{ on } [1.19R,+\infty), \quad \phi(1.1R)=1/2;$$
$$0\le \phi\le 1,\quad -c_1/R\le \phi'\le 0 \text{ on } [R,1.2R]$$
$$ -c_2/R^2\le \phi'' \le 0 \text{ on } [R,1.1R],$$ 
$$0 \le \phi'' \le c_2/R^2 \text{ on } [1.1R,1.2R].$$
We build a new function $h_s$ on $[R,1.2R]$ by 
$$h_s=\phi h_1 +(1-\phi)h_{i,+}$$
It is clear that $h_{i,+}\le h_s \le h_1$ on $[R,1.2R]$.
We check that $h_s$ satisfies (\ref{eq:new}) on $[R,1.2R]$.
We have
$$h_s'=\phi'(h_1-h_{i,+})+\phi h_1'+(1-\phi)h'_{i,+}.$$
It is clear that $h'_s<0$ on $[R,1.2R]$ because $h_1-h_{i,+}\ge 0$. On $[R,1.2R]$, there is a positive constant $c_3$ independent of $R$ such that
$$c_3^{-1} \le \left|\dfrac{h^{(j)}_1}{h^{(j)}_{i,+}}\right| \le c_3 $$
 holds for $j=0,1,2$. On the first half $[R,1.1R]$, $|h_s'|\ge \frac{1}{2}|h_1'|$; hence
$$\left|\dfrac{h'_s}{h_s}\right|\ge \dfrac{1}{2}\left|\dfrac{h_1'}{h_s}\right|\ge \dfrac{1}{2}\left|\dfrac{h_1'}{h_1}\right|.$$
On the second half $[1.1R,1.2R]$, we have $|h_s'|\ge \frac{1}{2}|h'_{i,+}|> \frac{1}{2c_3}|h_1'|$; hence 
$$\left|\dfrac{h_s'}{h_s}\right|\ge \dfrac{1}{2c_3}\left|\dfrac{h_1'}{h_1}\right|.$$
Next, we verify the second derivative estimate in $(\ref{eq:new})$. We calculate
$$h_s''=\phi''(h_1-h_{i,+})+2\phi'(h'_1-h'_{i,+})+\phi h_1''+(1-\phi)h_{i,+}''.$$
On $[R,1.1R]$,
$$h''_s \le 2 \dfrac{c_1}{R}|h'_1| + h_1''+h_{i,+}''\le c_4 h''_1;$$
on $[1.1R,1.2R]$,
$$h''_s\le \dfrac{c_2}{R^2} h_1+ 2\dfrac{c_1}{R}|h'_1| + h''_1+h''_{i,+}\le c_5 h''_1.$$
Together with $c_3^{-1}h_1\le h_{i,+}\le h_s$, we obtain
$$\dfrac{h''_s}{h_s}\le c_6 \dfrac{h''_1}{h_1}.$$
Therefore, we can replace $h$ from Step 1 on $[R,1.2R]$ by
$$h_s(r) =\begin{cases}
    h(r), & \text{if } r\not\in [R,1.2R],\\
    \phi h_1 +(1-\phi)h_{i,+}, & \text{if } r\in [R,1.2R].
\end{cases}$$
Then $h_s$ is smooth on $[0.5R,1.5R]$. Moreover, by Observation \ref{obs:Ric}, the resulting doubly warped product has positive Ricci curvature given $k$ is sufficiently large. This completes the smoothing at $R_{i,0}$.

The smoothing near other $R_{i,j}$ are similar. For $R=R_{i,2}$ or $R_{i,3}$, we change $h$ on $[0.8R,R]$ by using a suitable cut-off function. For $R=R_{i,4}$, we similarly change $h$ on $[R,1.2R]$. The choice of these intervals assures that the new function
$h_s$ is strictly decreasing. Then by Observation \ref{obs:Ric} and similar estimates as in the smoothing at $R_{i,0}$, one can show that the doubly warped product (\ref{eq:warp}) with this new $h_s$ has positive Ricci curvature for a sufficiently large $k$.

After the smoothing around all $R_{i,j}$ and the choice of a suitably large $k$, we obtain a smooth function $h_s$ with (\ref{eq:warp_h}) such that the doubly warped product (\ref{eq:warp}) with $h_s$ has positive Ricci curvature everywhere. Moreover, by construction
$$h_s(r)=\begin{cases}
   (1+r^2)^{-\alpha}, & \text{if } r\in [1.2R_{i,0},0.8R_{i,1}],\\
   (1+r^2)^{-\beta}, & \text{if } r\in [1.2R_{i,2},0.8R_{i,3}].
\end{cases}$$
This completes the construction of our example for Theorem \ref{thm:exmp}.

\subsection{Asymptotic properties}
We check that our example has the required asymptotic properties.

We fix a point $p\in M$ at $r=0$ and a lift $\tilde{p}\in \widetilde{M}$. Let $\gamma$ be a generator of $\pi_1(M,p)=\mathbb{Z}$. In the lemma below, we follow the method in \cite[Lemma 1.1]{PanWei} and \cite[Appendix B]{Pan21} to estimate $d(\gamma^l \tilde{p},\tilde{p})$; some modifications are needed given the nature of $h_s$.

\begin{lem}\label{lem:length_scale}
 There are positive constants $C_1,C_2,\rho_1,\rho_2$, depending on $\alpha$ or $\beta$, such that the following holds for all $i$ large.\\
 (1) For $l\in [\rho_1 R^{2\alpha+1}, \rho_2 R^{4\alpha+2}]$, where $R=2R_{i,0}$, we have
 \begin{equation*}
 C_1 l^{\frac{1}{2\alpha+1}} \le  d(\gamma^l \tilde{p},\tilde{p})\le   C_2 l^{\frac{1}{2\alpha+1}}.
\end{equation*}
(2) For $l\in [\rho_1 R^{2\beta+1}, \rho_2 R^{4\beta+2}]$, where $R=2R_{i,2}$, we have
\begin{equation*}
 C_1 l^{\frac{1}{2\beta+1}} \le  d(\gamma^l \tilde{p},\tilde{p})\le   C_2 l^{\frac{1}{2\beta+1}}.
\end{equation*}
\end{lem}
\begin{proof}
Let $c_l$ be a minimal geodesic loop at $p$ that represents $\gamma^l \in \pi_1(M,p)$. Below we estimate the length of $c_l$ for suitable $l$. Let $R_l$ be the size of $c_l$, that is, the smallest radius such that $c_l$ is contained in $\overline{B_{R_l}}(p)$, and let $r_l>0$, which will be specified later. From the proof of \cite[Lemma 1.1]{PanWei}, we have inequality
$$l\cdot L(R_l)\le \length(c_l)\le 2r_l + l\cdot L(r_l),$$
where $L(r)=2\pi h_s(r)$ is the length of a circle in $M$ at distance $r$. More precisely, the upper bound $2r_l + l\cdot L(r_l)$ comes from a loop obtained by joining a radial segment from $p$ to distance $r_l$, winding around the circle $l$ many times at distance $r_l$, and the inverse of the radial segment; the lower bound $l\cdot L(R_l)$ comes from winding around the circle $l$ many times at distance $R_l$. 

(1) We write $R=2R_{i,0}$. Recall that for $i\ge 2$, we have chosen $R_{i,1}=5R_{i,0}^2$ in Step 1. Then $[R,R^2]\subseteq[2R_{i,0},0.8R_{i,1}]$ and it follows from the construction of $h_s$ that
$$L(r_l)=2\pi (1+r_l^2)^{-\alpha} \le 2\pi r_l^{-2\alpha}$$
when $r_l\in [R,R^2]$. Let us consider $l\in [\frac{R^{2\alpha+1}}{2\pi\alpha},\frac{R^{4\alpha+2}}{2\pi\alpha}]$ and $r_l=(2\pi\alpha l)^{\frac{1}{2\alpha+1}}\in [R,R^2]$. Then
$$\length(c_l)\le 2r_l + l\cdot 2\pi r_l^{-2\alpha}=C(\alpha) l^{\frac{1}{2\alpha+1}},$$
where $C(\alpha)=(2+\alpha^{-1})(2\pi \alpha)^{\frac{1}{2\alpha+1}}$.

Next, we prove a lower bound for $\length(c_l)$. First note that
$$2R_l \le \length(c_l) \le C(\alpha) l^{\frac{1}{2\alpha+1}}.$$
If we further restrict that $l\le \frac{R^{4\alpha+2}}{C(\alpha)^{2\alpha+1}}$, then
$$R_l \le \dfrac{1}{2}C(\alpha) l^{\frac{1}{2\alpha+1}} \le \dfrac{1}{2}R^2.$$
We also note that if $l\ge C(\alpha)^{\frac{2\alpha+1}{2\alpha}} R^{2\alpha+1}$, then $R_l\ge R$ holds; otherwise, we would have
$$C(\alpha) l^{\frac{1}{2\alpha+1}}\ge l\cdot L(R_l)\ge l\cdot L(R)\ge l\cdot 2\pi(1+R^2)^{-\alpha},$$
$$\dfrac{C(\alpha)}{2\pi} (1+R^2)^\alpha \ge l^{\frac{2\alpha}{2\alpha+1}}\ge C(\alpha) R^{2\alpha},$$
which is impossible given $R$ large.
Since $R_l\in [R,R^2]\subseteq [2R_{i,0},0.8R_{i,1}]$, 
$$l\cdot 2\pi(1+R_l^2)^{-\alpha}=l\cdot L(R_l) \le C(\alpha) l^{\frac{1}{2\alpha+1}}$$
yields
$$R_l \ge \dfrac{1}{2}\left( \dfrac{2\pi}{C(\alpha)} \right)^{\frac{1}{2\alpha}} l^{\frac{1}{1+2\alpha}}.$$
This shows the lower bound
$$\length(c_l)\ge 2R_l \ge \left( \dfrac{2\pi}{C(\alpha)} \right)^{\frac{1}{2\alpha}} l^{\frac{1}{1+2\alpha}}.$$

In conclusion, we set 
$$\rho_1= C(\alpha)^{\frac{2\alpha+1}{2\alpha}},\quad \rho_2=\dfrac{1}{C(\alpha)^{2\alpha+1}},\quad C_1=\left( \dfrac{2\pi}{C(\alpha)} \right)^{\frac{1}{2\alpha}},\quad C_2=C(\alpha),$$
then
$$C_1 l^{\frac{1}{2\alpha+1}} \le d(\gamma^l \tilde{p},\tilde{p}) \le C_2 l^{\frac{1}{2\alpha+1}}$$
holds for all $l\in [\rho_1 R^{2\alpha+1},\rho_2 R^{4\alpha+2}]$.

The proof for (2) is similar.
\end{proof}

With Lemma \ref{lem:length_scale}, we prove the orbit growth order in Theorem \ref{thm:exmp}.

\begin{proof}[Proof of Theorem \ref{thm:exmp}] 
We first note that $\length(c_l)\leq \length(c_{l+1})$ for every $l\in \mathbb{N}$. Let $t_i=2R_{i,0}$ and $b_i=\lceil \frac{\rho_2}{2} t_i^{4\alpha+2}\rceil$. By Lemma \ref{lem:length_scale}, for $r_i=C_1b_i^{\frac{1}{2\alpha+1}}$, we have $\length(c_{b_i})\geq  r_i$. Thus 
$$\#\{l\in\mathbb{N}_+:d(\Tilde{p},\gamma^l\Tilde{p})\leq r_i\}\leq b_i=\left(\frac{r_i}{C_1}\right)^{2\alpha+1}.$$
Next, we set $a_i= \lfloor(\frac{r_i}{C_2})^{2\alpha+1}\rfloor=\lfloor(\frac{C_1}{C_2})^{2\alpha+1}b_i \rfloor$. For $i$ large, again by Lemma \ref{lem:length_scale}, we have $\length(c_{a_i})\leq r_i$, so 
$$\#\{l\in\mathbb{N}_+:d(\Tilde{p},\gamma^l\Tilde{p})\leq r_i\}\geq a_i\ge \left(\frac{r_i}{2C_2}\right)^{2\alpha+1}.$$
We complete the proof for (1), that is,
$$c_1 r_i^{2\alpha+1}\le \#\Gamma(r_i) \le c_2 r_i^{2\alpha+1}.$$

The proof for (2) is similar.
\end{proof}
 
\begin{rem}\label{rem:limit_Grushin}
Similar to the argument in \cite[Section 3]{DHPW}, we can use a change of variables to show that both $(2\alpha)$- and $(2\beta)$-Grushin halfplanes appear as asymptotic cones of $\widetilde{M}$, the universal cover of $M$ in Theorem \ref{thm:exmp}. In fact, $\widetilde{M}$ is a doubly warped product
$$\widetilde{M}=[0,\infty)\times_f S^k \times_{h_s} \mathbb{R},\quad  \tilde{g}=dr^2+\dfrac{r^2}{(1+r^2)^{1/2}} ds_k^2 + h_s(r)^2 dv^2.$$
On $[1.2R_{i,0},0.8R_{i,1}]$, where $R_{i,1}\gg R_{i,0}$, we have $h_s(r)=(1+r^2)^{-\alpha}$. We apply a changes of variables $t=\lambda^{-1}r$ and $w=\lambda^{-1-2\alpha} v$, where $\lambda>1$. Then
$$\lambda^{-2}\tilde{g}=dt^2 + \dfrac{t^2}{(1+\lambda^2 t^2)^{1/2}}ds_k^2 + h_s(\lambda t) \lambda^{4\alpha} dw^2.$$
When $t\in [\lambda^{-1} 1.2R_{i,0},\lambda^{-1} 0.8R_{i,1}]$, we have
$$h_s(\lambda t) \lambda^{4\alpha}=\left( \dfrac{\lambda^2}{1+\lambda^2t^2} \right)^{2\alpha} \to t^{-4\alpha}$$
as $\lambda\to\infty$. Let us choose a sequence $\lambda_i\to\infty$ with $R_{i,0}\ll \lambda_i \ll R_{i,1}$, then $[\lambda_i^{-1} 1.2R_{i,0},\lambda_i^{-1}0.8R_{i,1}]$  exhausts $[0,\infty)$ as $i\to\infty$. Therefore, we have convergence
$$\lambda_i^{-1}\tilde{g} \to dt^2 + t^{-4\alpha} dw^2.$$
This shows that $(\lambda_i^{-1}\widetilde{M},\tilde{p})$ Gromov-Hausdorff converges to the $(2\alpha)$-Grushin halfplane as $i\to\infty$. Similarly, for $R_{i,2} \ll \lambda_i \ll R_{i,3}$, $(\lambda_i^{-1}\widetilde{M},\tilde{p})$ converges to the $(2\beta)$-Grushin halfplane.
\end{rem}

One can modify the construction in Theorem \ref{thm:exmp} to obtain 

\begin{thm}\label{thm:vary_Hdim}
   Given $2\le d_1<d_2$, there is an open $n$-manifold $N$ of positive Ricci curvature, where $n$ is sufficiently large, such that for any real number $d\in [d_1,d_2]$, there is an asymptotic cone of $N$ with Hausdorff dimension $d$.
\end{thm}

\begin{proof}
   Let $\beta_j\ge 1/2$ such that $1+2\beta_j=d_j$, where $j=1,2$. We arrange all the rational numbers in $[\beta_1,\beta_2]$ into a sequence $\{\alpha_i\}$ such that any real number in $[\beta_1,\beta_2]$ is an accumulation point of $\{\alpha_i\}$. Let 
   $$h_i(r)=(1+r^2)^{-\alpha_i}.$$
   Following a similar construction for Theorem \ref{thm:exmp}, we can have scales $R_{i,1}$ and $R_{i,2}$, where $i\in\mathbb{N}_+$, such that
   $$R_{i,1}\ll R_{i,2} \ll R_{i+1,1},$$
   and a smooth function $h_s:[0,\infty)\to \mathbb{R}^+$ with (\ref{eq:warp_h}) such that 
   $$h_s=h_i \text{ on } [1.2R_{i,1},0.8R_{i,2}].$$
   Then for a large $k$, the doubly warped product 
   $$M=[0,\infty)\times_f S^k \times_{h_s} S^1,\quad g=dr^2+f^2(r)ds_k^2 + h_s^2(r)ds_1^2$$
   has positive Ricci curvature.

   Let $N=\widetilde{M}$ be the Riemannian universal cover of $M$. For any $d\in [d_1,d_2]$, we can choose a subsequence $\alpha_{i(j)}$ of $\alpha_i$ such that 
   $$1+2\alpha_{i(j)}\to 1+2\alpha=d.$$
   Then by the same argument in Remark \ref{rem:limit_Grushin}, for any sequence $\lambda_{i(j)}\to \infty$ with $R_{i(j),1}\ll \lambda_{i(j)} \ll R_{i(j),2}$, the corresponding asymptotic cone of $N$
   $$(\lambda_{i(j)}^{-1}N,\tilde{p})\overset{GH}\longrightarrow (Y,y)$$
   is the $(2\alpha)$-Grushin plane with Hausdorff dimension $d$.
\end{proof}

\bibliographystyle{plain} 
\bibliography{refs}
\end{document}